\documentclass[12pt]{article}
\usepackage{amsfonts}

\usepackage{latexsym,amsfonts,amsmath,amssymb,array,epsfig,tabularx,amsthm,dsfont,enumerate}
\usepackage{multirow}
\usepackage{verbatim}
\usepackage{caption}
\usepackage{float}

\makeatletter
\newcommand{\LeftEqNo}{\let\veqno\@@leqno}
\makeatother

\theoremstyle{definition}
\newtheorem{thm}{Theorem}
\newtheorem{rem}[thm]{Remark}
\newtheorem{lem}[thm]{Lemma}
\newtheorem{prop}[thm]{Proposition}
\newtheorem{cor}[thm]{Corollary}

\newcommand\R{{\mathbb{R}}}
\newcommand\C{{\mathbb{C}}}

\newcommand\N{\mathbb{N}}
\newcommand\Z{\mathbb{Z}}

\newcommand{\wt}{\widetilde}
\newcommand{\widebar}[1]{\mbox{\kern1pt\hbox{\vbox{\hrule height 0.5pt \kern0.25ex
        \hbox{\kern-0.05em \ensuremath{#1 }\kern-0.05em}}}}\kern-0.1pt}

\newcommand{\dx}{\, {\rm d} x }

\newcommand{\abs}[1]{\left\vert #1 \right\vert}
\newcommand{\norm}[1]{\left\Vert #1 \right\Vert}

\newlength{\fixboxwidth}
\renewcommand{\rho}{{\varrho }}

\newcommand{\PreserveBackslash}[1]{\let\temp=\\#1\let\\=\temp}
\newcolumntype{R}[1]{>{\PreserveBackslash\raggedleft}p{#1}}
\title{Optimal Quadrature Formulas for the Sobolev Space $H^1$}

\author{
Erich Novak\footnote{This
author was partially supported by the DFG-Priority Program 1324.},\\
Mathematisches Institut, Universit\"at Jena\\
Ernst-Abbe-Platz 2, 07743 Jena, Germany\\
email: erich.novak@uni-jena.de \\
Shun Zhang\footnote{This author was partially
supported by the National Natural Science Foundation of China
(Grant No. 11301002).
}\\
School of Computer Science and Technology, Anhui
University,\\
Hefei 230601, China\\
email:\ shzhang27@163.com}

\begin{document}
\maketitle

\centerline{Dedicated to Henryk Wo\'zniakowski on the occasion of his 70th birthday}

\begin{abstract}
We study optimal quadrature formulas for arbitrary
weighted integrals and integrands from the Sobolev space $H^1([0,1])$.
We obtain general formulas for the worst case error depending on
the nodes $x_j$.
A particular case is the computation of Fourier coefficients,
where the oscillatory weight
is given by $\rho_k(x) = \exp(- 2 \pi i k x)$.
Here we study the question whether equidistant nodes are optimal
or not. We prove that this depends on $n$ and $k$:
equidistant nodes are \emph{optimal}
if $n \ge  2.7 |k| +1 $ but might be suboptimal for small $n$.
In particular, the equidistant nodes
$x_j = j/ |k|$ for $j=0, 1, \dots , |k| = n+1$ are the \emph{worst}  possible nodes
and do not give any useful information.
To characterize the worst case function we use certain
results from the theory of weak solutions of boundary value problems
and related quadratic extremal problems.
\end{abstract}

\section{Introduction}

We know many results about optimal quadrature formulas, see
Brass and Petras \cite{BP11} for a recent monograph.
This book also contains important results
for the approximation of Fourier coefficients of periodic functions,
mainly for equidistant nodes.
As a general survey for the computation of oscillatory integrals
we recommend
Huybrechs and Olver~\cite{HO09}.

It follows from results of
\cite{NUW15,NUWZ15} that equidistant nodes lead to quadrature formulas that are
asymptotically optimal
for the standard
Sobolev spaces $H^s([0,1])$ of periodic functions and also for $C^s$ functions.

We want to know whether equidistant nodes are optimal or not.
${\rm\check{Z}}$ensykbaev \cite{Zen77} proved that for the classical (unweighted)
integrals of periodic functions from the Sobolev space $W^r_p([0,1])$,
the best quadrature formula of the form $A_n(f)=\sum_j a_j f(x_j)$
is the rectangular formula with equidistant nodes.
Algorithms with equidistant nodes were also studied by Boltaev, Hayotov and Shadimetov~\cite{BHS16b}
for the numerical calculation of Fourier coefficients.
In their paper not only the rectangular formula was studied but all (and optimal)
%S ^^^^ instead of this
formulas based on equidistant nodes.
It is not clear however, whether equidistant nodes are optimal or not.

We did not find a computation of the worst case error of optimal quadrature formulas
for general $x_j$
in the literature.
Related to this,
we did not find a discussion about whether equidistant nodes are optimal for oscillatory integrals or not.
We find it interesting that the results very much
depend on the frequency of oscillations and the number of nodes.
%S Aug.14:    ^^^^^^^^^^^^^^^^^^^ means the number k? replaced by "frequency of oscillations"?
%E  ok, I changed it.
%S Aug. 31 Agreed.

This paper has two parts.
In the first part,
we present general formulas for the worst case error for arbitrary weighted
integrals in the Sobolev space $H^1$ for arbitrary nodes.
In the second part we consider oscillatory integrals
and prove that equidistant nodes
are optimal for relatively large $n$,
but can be very bad for small $n$.

We now describe our results in more detail.
We study optimal algorithms
for the computation of integrals
\[
I_\rho(f)=\int_0^1 f(x) \rho(x) \dx ,
\]
where the density $\rho$ can be an arbitrary integrable function.
We assume that the integrands are from the Sobolev space $H^1$ and, for simplicity,
often assume zero boundary values, i.e., $f \in H_0^1$.
Here $H_0^1 = H^1_0 ([0,1])$ is the space of all absolutely continuous
functions with values in $\C$ such that $f^\prime \in L_2$
and $f(0)=f(1)=0$. The norm in $H^1_0$
(and semi-norm in $H^1$)  is given by
$\Vert f \Vert := \Vert f^\prime\Vert_{L_2}$.
We simply write $\Vert \cdot \Vert_2$ instead of $\Vert \cdot \Vert_{L_2}$.

We study algorithms that use a ``finite information''
$N: H^1 \to \C^n$ given by
\[
N(f)= (f(x_1),\dots,f(x_n)) .
\]
We may assume that
\[
  0 \le x_1 < x_2 < \dots < x_n \le 1.
\]

We prove general results for the worst case error for arbitrary
$\rho$ and nodes $(x_j)_j$ and then study in more detail
integrals with the density function
$\rho_k(x) = \exp(- 2 \pi i k x)$.
Here we want to know whether equidistant nodes
%E  $x_j = j/(n+1)$
%E  I want to skip this formula since equidistance has (slightly) different meanings, depending on
%E  boundary conditions.
%S Aug. 31 Agreed.
are optimal or not.
We shall see that this depends on $n$ and $k$:
equidistant nodes are optimal
if $n \ge  2.7 |k| + 1$ but might be suboptimal for small $n$.
In particular, the equidistant nodes
$x_j = j/ |k|$ for $j=0, 1 \dots , |k| = n+1$ are the \emph{worst}  possible nodes
and do not give any useful information.
To characterize the worst case function we use certain
results from the theory of weak solutions of boundary value problems
and related quadratic extremal problems.

The aim of this paper is to prove some exact formulas on the
$n$th minimal (worst case) errors, for $F\in\{H^1_0,H^1\}$,
\[
e(n, I_\rho , F) :=
\inf_{A_n}\, \sup_{f\in F:\, \|f\|_F\le1}\, |I_\rho (f) - A_n(f)|.
\]
This number is
the worst case error on the unit ball of $F$
of an optimal algorithm $A_n$ that uses at most $n$ function values
for the approximation of the functional $I_\rho$.
The initial error is given for $n=0$ when we do not sample the
functions. In this case the best we can do is to take
the zero algorithm $A_0(f)=0$,
and
\[
e(0, I_\rho , F) :=
\sup_{f\in F:\, \|f\|_F\le1}\, |I_\rho (f)| = \| I_\rho \|_{F}.
\]

%E We could mention our main results here instead of a
%E Section 5: Conclusions.  Is it ok?
%E Please check!
%S Aug. 31 Agreed.

Let us collect the main results of this paper:
\begin{enumerate}[(i)]
\item
For general (integrable) weight functions
$\rho : [0,1 ] \to \C$, we derive formulas
for the initial error
(Proposition~\ref{prop1}) and for the  radius of information
(worst case error of the optimal algorithm)
for arbitrary nodes
(Theorem~\ref{thm2}).

\item
We study oscillatory integrals with the weight
function
$\rho_k(x) = \exp(- 2 \pi i k x)$
for the space $H^1_0([0,1])$.
In Proposition~\ref{prop2}
we compute the
initial error
for  $k\in\Z\backslash\{0\}$
%S Aug.   ^^ instead of \R
and the main result  is  Theorem~\ref{thm3} for  $k\in\R\backslash\{0\}$,
%S Aug.                                added ^^^^^^^^^^^^^^^^^^^^^^^^^^^
where we prove that equidistant nodes are optimal if
$n \ge 2.7 |k|-1$.

\item
Then we study the full space
$H^1([0,1])$ and again prove that equidistant nodes are
optimal for $k\in\R\backslash\{0\}$ and large $n$.
%S Aug.: ^^^^^^^^^^^^^^^^^^^^^^^^^^ added
See
Theorem~\ref{thm5} for the details.
We could prove very similar results also for the subspace of $H^1([0,1])$ of  periodic functions
or for functions with a boundary value (such as $f(0) =0$).
Since the results and also the proofs are similar, we skip the details.

%E  the next one is new, please check. more details needed?
%S Aug.:  checked already.
\item
In Section 4 we discuss results for equidistant nodes $x_j = j/n$, for $j = 0, 1, \dots , n$,
and prove certain asymptotic results (which are the same for equidistant and optimal nodes).
In particular we obtain
%S Aug. 11:  add ", similar to the space $H^1_0$,"?
$$
\lim_{|k| \to \infty} e(n, I_{\rho_k} , H^1) \cdot  |k| =  \frac{1}{2 \pi}
$$
for each fixed $n$ and
$$
\lim_{n \to \infty} e(n, I_{\rho_k} , H^1 ) \cdot n = \frac{1}{2 \sqrt3}
$$
for each fixed $k\in\R\backslash\{0\}$.
%^^^^^^^^^^^^^^^^^^^^^^^^^^^^^^^^^^^^^^^^^^^^^^^^^^^
%S Aug. 12:  proved also for H_0^1 even with more general nodes.
%            See Cor. 10 and  last Remark 18.
\end{enumerate}

\section{Arbitrary density functions}\label{Se:arbitrary}

We start with
\[
I_\rho(f)=\int_a^b f(x) \rho(x) \dx
\]
for
$f\in H_0^1([a,b])$ and want to compute
the so called initial error
\[
e_0 : = \sup_{\Vert f \Vert \le 1 } |  I_\rho (f) |  .
\]
Since the complex valued case is considered here,
the inner product in the spaces $H_0^1([a,b])$ is given by
\[
\langle f,g \rangle = \int_a^b f^\prime(x) \overline{g^\prime(x)} \dx .
\]
Using the integration by parts formula we see that
the initial error is given by
\[
e_0 =
\sup\limits_{\norm{f^\prime}_2\le 1 \atop \int_a^b f^\prime =0} \abs{\int_a^b f^\prime(x) \cdot R(x) \dx } ,
\]
where  $R(t)= \int_a^t \rho(x) \dx$  for $t \in [a,b]$.
To solve the extremal problem
$$
\sup\limits_{\norm{g}_2\le1 \atop \int_a^b g=0} \abs{  \int_a^b g(x)R(x) \dx },
$$
we decompose $R$ into a constant $c$ and an orthogonal
function $\wt R$,
%S Aug.:  ^^^ this was \tilde. replaced for all.
$R=\wt R + c$, hence $c=\frac1{b-a} \int_a^b R(x) \dx,\
\wt R = R-c$ and $\int_a^b \wt R(x) \dx = 0$.
It then follows from the Cauchy-Schwarz inequality that
every
$g^\ast= \gamma  \frac{\wt R}{\Vert \wt R\Vert_2}$
with
$|\gamma| = 1$
solves the extremal problem and the respective maximum is $\norm{R-c}_2 = \Vert \wt R \Vert_2 $.

We define $f^*$ by
\[
f^*(t) = - \int_a^t \frac{\overline{R(x)}-\overline{c}}{\norm{R-c}_2} \dx .
\]
Then $f^*(a) = f^*(b) = 0$ and $f^* \in H^1_0 ([a,b])$. Further,
\[
\begin{split}
\int_a^b f^* (x) \rho(x) \dx
&= \, \int_a^b f^* (x) \, {\rm d} R(x)\\
&= \, f^* (x) R(x)|_a^b - \int_a^b (f^*)'(x) R(x) \dx
\\
&= \, - \int_a^b (f^*)'(x) (R(x)-c) \dx
\\
&= \, \int_a^b \frac{\overline{R(x)}-\overline{c}}{\norm{R-c}_2} \big( R(x)-c \big) \dx
\\
&= \, \norm{R-c}_2.
\end{split}
\]

\begin{rem}\label{re:Ra}
It is easy to check that the property $R(a)=0$ is not used
in the above computations.
Therefore it is not important what is chosen as
the lower limit of the integral in the definition of $R$.
\end{rem}

Hence we have proved the following proposition.

\begin{prop}   \label{prop1}
Consider $I_\rho: H^1_0 ([a,b]) \to \C$
with an integrable density function $\rho$.
Then
\[
e_0 =  \sup_{\Vert f \Vert \le 1 } | I_\rho (f) | =
\Vert R-c \Vert_2 ,
\]
where
$R(t)= \int_a^t \rho(x) \dx$  for $t \in [a,b]$
and
$c=\frac1{b-a} \int_a^b R(x) \dx$.
Moreover the maximum is assumed for $f^* \in H^1_0 ([a,b])$, given by
\[
f^*(t) = - \int_a^t \frac{\overline{R(x)}-\overline{c}}{\norm{R-c}_2} \dx ,
\]
i.e., $I_\rho (f^*) = \Vert R-c \Vert_2$
and
$\Vert f^* \Vert=1$ with $f^*(a)=f^*(b)=0$.
\qed
\end{prop}

The initial error $e_0$ clearly depends
on $a$, $b$ and $\rho$ and later we will write
$e_0(a,b,\rho)$ for it.

We are in a Hilbert space setting
(with the two Hilbert spaces $H=H^1([0,1])$ and
$H^1_0([0,1])$) and the structure
of optimal algorithms
$A = \phi \circ N$, for a given
information $N: H \to \C^n$, is known:
the spline algorithm is optimal and the spline
$\sigma$ is continuous and piecewise linear,
see \cite[Cor. 5.7.1]{TWW88} and \cite[p. 110]{TW80}.

More exactly, if $N(f)= y \in \C^n $ are the function values
at $(x_1, \dots , x_n)$, then
$A(f) = \phi(y) = I_\rho( \sigma)$.
In the case $H=H_0^1([0,1])$ the spline
$\sigma$ is given by
$\sigma(0)=\sigma(1)=0$ and $\sigma(x_i) = f(x_i) = y_i$
and piecewise linear.
In the case $H=H^1([0,1])$ the spline is constant
in $[0,x_1]$ and $[x_n, 1]$, otherwise it is the same
function as in the case $H=H_0^1([0,1])$.

Moreover, we have the general formula for the worst case error
of optimal algorithms $A$
\[
\sup_{\Vert f \Vert_H \le 1} |I_\rho(f) - A(f) | =
\sup_{\Vert f \Vert_H \le 1, \ N(f)=0}  | I_\rho (f) |.
\]
This number is also called the radius $r(N)$ of the information $N$
and to distinguish the two cases,
we also write $r(N, H^1)$ and $r(N, H^1_0)$, respectively,
see \cite[Thm. 5.5.1 and Cor. 5.7.1]{TWW88} and \cite[Thm. 2.3 of Chap. 1]{TW80}.

We are ready to present a general formula
for $r(N, H^1_0)$ and afterwards solve another
extremal problem to present the formula for
$r(N, H^1)$.

We put $x_0=0$ and $x_{n+1}=1$ and then have $n+1$ intervals $I_j = [x_j, x_{j+1}]$,
where $j=0,1,\cdots,n$.
For the norm $\norm{f}:= \norm{f^\prime}_2$, the worst case function $f^*_j$  is, on any interval $I_j$,
as in Proposition~\ref{prop1}.
The norm of $f^*_j$ is one and the integral is
$e_0(x_j,x_{j+1},\rho)=: c_j$.
Then the radius of information of the information $N$  is given by
$$
r(N) = \max\limits_{\alpha_j \ge 0 \atop \sum \alpha_j^2=1} \sum_j
\alpha_j \,   e_0( x_j, x_{j+1}, \rho)
$$
and it is easy to solve this extremal problem.
The
maximum is taken for
$\alpha_j = (\sum_j c_j^2)^{-1/2} c_j$
and then the total error is the radius of information,
$r(N)=\sum_j \alpha_j c_j =  (\sum_j c_j^2)^{1/2}$.
As a result we obtain the following assertion.

\begin{thm}   \label{thm2}
In the case of $H^1_0([0,1])$ the radius of information is given by
\[
r(N) = \left(  \sum_{j=0}^n e_0 (x_j, x_{j+1}, \rho)^2 \right)^{1/2} .
\]
Moreover, the worst case function $f^*$
is given by
\[
f^\ast|_{I_j} = \left( \sum\limits_{j=0}^n c_j^2 \right)^{-1/2} \cdot c_j \cdot f^*_j ,
\]
where $c_j = e_0 (x_j, x_{j+1}, \rho)$.
In particular we have
$f^* \in H^1_0 ([0,1])$ with norm 1 and
$N(f^*)=0$ with $I_\rho (f^*) = r(N)$.
\qed
\end{thm}

Now we turn to the space $H^1([0,1])$. In this case, we need a small modification for the intervals
$[0,x_1]$ and $[x_n,1]$ since the value of $f(0)$ is unknown if $x_1 > 0$ and $f(1)$ is unknown if $x_n < 1$.

For those functions $f\in H^1([a,b])$ satisfying $f(a)=0$, we take $R(t)=\int_t^b \rho(x) \dx$,\ $t\in[a,b]$.
Then $R(b)=0$ and  the respective maximum is $\norm{R}_2$. We define $f^\ast \in H^1([a,b])$ by
\[
f^*(t) =  \int_a^t \frac{\overline{R(x)}}{\norm{R}_2} \dx.
\]
Then $f^*(a) = 0,\ (f^*)'(b)=0$ and $\norm{f^*}=1$. Afterwards,
\begin{equation}\label{eq:initial-R}
\begin{split}
\int_a^b f^* (x) \rho(x) \dx
&= \, -\int_a^b f^* (x) \, {\rm d} R(x)\\
&= -\, f^* (x) R(x)|_a^b + \int_a^b (f^*)'(x) R(x) \dx
\\
&= \,  \int_a^b (f^*)'(x) R(x) \dx
\\
&= \, \int_a^b \frac{\overline{R(x)}}{\norm{R}_2}  R(x)  \dx
\\
&= \, \norm{R}_2.
\end{split}
\end{equation}

Similarly, for the functions $f\in H^1([a,b])$ satisfying $f(b)=0$, we take $R(t)=\int_a^t
\rho(x) \dx$,\ $t\in[a,b]$. Then $R(a)=0$ and the respective maximum is $\norm{R}_2$. We define $f^\ast$ by
\[
f^*(t) = - \int_t^b \frac{\overline{R(x)}}{\norm{R}_2} \dx.
\]
Then $f^*(b) = 0,\ (f^*)'(a)=0$ and $\norm{f^*}=1$. Also,
$I_\rho(f^*)=  \norm{R}_2$.

Hence, we obtain almost the same assertion for the full space $H^1([0,1])$ as in Theorem~\ref{thm2}.
Here, $c_0=e_0 (0, x_{1}, \rho)=\norm{R}_2$ on $[0,x_1]$ if $x_1>0$ and $c_n=e_0 ( x_{n}, 1, \rho)
=\norm{R}_2$ on $[x_n,1]$  if $x_n<1$, instead of so-called $\norm{R-c}_2$. Accordingly, $f^\ast_0$
and  $f^\ast_n$ should be changed.

Observe that the initial error is infinite if $I(\rho) \not= 0$
since all constant functions have a semi-norm zero. Therefore we now assume that $I(\rho)=0$.
Then for the full space $H^1([0,1])$,
the initial error of the problem $I_\rho$ is, as in \eqref{eq:initial-R},
\begin{equation}             \label{eq:initial-R-b}
e_0  \left( H^1,\rho \right) := \sup_{\Vert f \Vert_{H^1} \le 1 } | I_\rho (f) |
= \Vert R \Vert_2 ,
\end{equation}
where
$R(t)= \int_t^1 \rho(x) \dx$  for $t \in [0,1]$.

%E It seems that the following was a repetition?!
%S 22 July: Yes. Almost a repetition and we skip it.

%  For optimal algorithms $A_n$, we have $A_n(f(0))= I_\rho(f(0)) = A_0(f)$, and
%  \[
%  \begin{split}
%  I_\rho (f) - A_n(f)  & = \left(I_\rho (f) - A_n(f)\right) - \left( I_\rho (f(0)) - A_0(f(0))\right)\\
%  & =  \left(I_\rho (f) - A_n(f)\right) - \left( I_\rho (f(0)) - A_n(f(0))\right)
%  \\
%  &  =I_\rho \left( f - f(0)\right) -  A_n \left( f - f(0)\right).
%  \end{split}
%  \]
%
%  Under the assumption that $I(\rho)=0$ or the boundary value $f(0)$ is given, again we obtain
%  the same assertion for the full space $H^1([0,1])$ as in Theorem~\ref{thm2}.
%  Here, $c_0=e_0 (0, x_{1}, \rho)=\norm{R-c}_2$ on $[0,x_1]$ if $x_1>0$ and $c_n=e_0 ( x_{n}, 1, \rho)
%  =\norm{R}_2$ on $[x_n,1]$  if $x_n<1$.
%  If the boundary value $f(1)$ is given instead of $f(0)$,
%  then $R(t)= \int_0^t \rho(x) \dx$ for the initial error and a similar result is obtained.

\begin{rem}\label{re:Lax}
We can apply the theory of ``weak solution of elliptic boundary value problems
in the Sobolev Space $H^1$ and related extremal problems'' in the simplest case,
in particular Lax-Milgram Lemma (see \cite{Cia02,LM54}),
and we obtain the following fact:
The boundary value problem
\[
f^{\prime\prime} = -  \rho,\quad \quad f(a) = f(b)=0,
\]
is equivalent to the extremal problem of finding the minimizer
of the functional
\[
J( f ) = \frac{1}{2}  \norm{f^\prime}_2^2 - \int_a^b f  \rho\, \dx,\,
\]
where $f \in H_0^1([a,b])$.
Hence the minimizer of this extremal problem is the unique solution
of the boundary value problem.
For the space $H^1_0 ([0,1])$ which was considered in Theorem~\ref{thm2} this gives just
another proof of the same result.

If we now consider the full Sobolev space $H^1([0,1])$ then we obtain a slightly different
extremal problem in the first interval
$[a,b]=[0,x_1]$ and  in the last interval $[a,b]=[x_n,1]$.
The extremal problem for the first interval is:
Minimize $J$ as above where now $f$ is from the set of
$H^1 ([a,b])$ with  $f(b)=0$, while $f(a)$ is arbitrary.
It is well known and easy to prove that the respective
boundary value problem is
\[
f^{\prime\prime} = - \rho,\quad \quad f'(a) = 0, f(b)=0 ,
\]
and similar for the last interval.

With these modifications, we obtain a formula for the radius $r(N)$  of the information
for the space $H^1([0,1])$, the same formula as in Theorem~{\ref{thm2}}, only the numbers
$e_0(0,x_1, \rho)$ and
$e_0(x_n, 1, \rho )$ are defined differently, with the modified
extremal problem or modified boundary value problem.
\qed
\end{rem}

\begin{rem}
In the case $\rho_k (x)= \exp(-2 \pi ikx)$, the worst case function is, in each interval
$I_j= [x_j, x_{j+1}]$, of the form
\[
f(x)=c_j \exp(-2 \pi i k x) + a_j x +b_j,
\]
with $f(x_j) =0$ for $j=1, \dots, n$
and
$f'(0)=0$ if $x_1 >0$ and
$f'(1)=0$ if $x_n <1$.
\end{rem}

\section{Oscillatory integrals: optimal nodes}\label{Se:osci}

In this section we consider optimal nodes for integrals with the density function
\[
\rho_k(x) = \exp(-2 \pi i k x), \quad\quad k\in\R\backslash\{0\}, \quad\quad x \in [0,1].
%S Aug. 11                                     ^^ this was \Z.
\]
The integrands are from the spaces $H_0^1 ([0,1])$ or $H^1 ([0,1])$, respectively.

\subsection{The case with zero boundary values}   \label{Se:period}

We want to know whether in this case equidistant nodes, i.e.,
\[
x_j = \frac{j}{n+1}, \qquad j= 1, \dots , n,
\]
are optimal for the space
$H^1_0 ([0,1])$  or not. We will see that they are optimal for large $n$, but not for small $n$.

Following Section \ref{Se:arbitrary},
in this case we can consider a general interval $[a,b]$ and compute $R(x)$, constant $c$, and the initial
error $\Vert R-c \Vert_2$. Then we obtain that
the initial error depends only on $k$ and the length $L=b-a$ of the interval,
it is nondecreasing with $L$. We establish that
equidistant $x_j = \frac j{n+1}$ are optimal for large $n$ compared with $|k|$.

According to Remark \ref{re:Ra}, we modify the lower limit of the integral for $R(x)$ and define simply
\begin{equation*}
R(x):= \int_0^x \rho_k(t) {\rm d} t = \int_0^x e^{- 2\pi i k t} {\rm d} t = \frac{e^{- 2\pi i k x}-1}{- 2\pi ik},
\end{equation*}
and
\[
c := \frac 1{b-a} \int_a^b R(x) \dx= - \frac{e^{- 2\pi i k b}-e^{- 2\pi i k a}}{4\pi^2 k^2 L} + \frac1{2\pi ik}.
\]
Then on the interval $[a,b]$,
\begin{equation}
\begin{split}\label{eq:initial-R-c}
&\quad\ \norm{R-c}_2^2\\
&= \, \int_a^b \big( R(x)-c \big) \big( \overline{R(x)}-\overline{c} \big)  \dx \\
&= \, \int_a^b  R(x) \overline{R(x)} \dx  - L\, c\, \overline{c}\\
%S Aug. 11  skip the following two lines?
&= \, \int_a^b  \frac{e^{-2\pi i k x}-1}{-2\pi ik} \frac{e^{2\pi i k x}-1}{2\pi ik} \dx
- L
\left( - \frac{e^{-2\pi i k b}-e^{-2\pi i k a}}{4 \pi^2 k^2 L} + \frac1{2\pi ik}\right)\cdot
\left(- \frac{e^{2\pi i k b}-e^{2\pi i k a}}{4 \pi^2 k^2 L}- \frac1{2\pi ik}\right)
\\
&= \, \frac{1}{4 \pi^2 k^2} \int_a^b 2 \left(  1 - \cos(2 \pi k x) \right)\dx
- \frac{1 - \cos(2 \pi k L)}{8\pi^4 k^4 L} - \frac{- \sin(2 \pi k b) + \sin(2 \pi k a)}{4\pi^3 k^3} - \frac{L}{4 \pi^2 k^2}
\\
&= \, \frac{L}{4 \pi^2 k^2} - \frac{1}{8\pi^4 k^4 L}\left( 1- \cos(2\pi k L)\right),
\end{split}
\end{equation}
which is independent of $a$ and $b$ and stays the same even if $R(x):= \int_a^x e^{- 2\pi i k t} \,  {\rm d} t$.

{}From Proposition \ref{prop1}, we easily obtain the following assertion concerning the initial error.

\begin{prop}   \label{prop2}
Consider the oscillatory integral $I_{\rho_k}: H^1_0 ([0,1]) \to \C$ with $k\in\Z\backslash\{0\}$.
Then the initial error is given by
\[
e_0 =  \sup_{\Vert f \Vert \le 1 } \abs{ I_{\rho_k} (f) } =
\frac 1{2 \pi \abs{k}}.
\]
Moreover the maximum is assumed for $f^* \in H^1_0 ([0,1])$, given by
\[
f^*(t) = \frac 1{2 \pi \abs{k}} \left( e^{2\pi i k t}-1 \right),
\]
i.e., $I_{\rho_k} (f^*) = e_0$
and
$\Vert f^* \Vert=1$ with $f^*(0)=f^*(1)=0$.
\qed
\end{prop}

Following Theorem~\ref{thm2},
denote $L_j=|I_j|$, then $\sum\limits_{j=0}^n L_j =1$ and $c_j=\norm{R-\beta_j}_2$ with
$\beta_j = \frac 1{L_j} \int_{I_j} R(x) \dx$.
The radius of information is
\[
\left(\sum_{j=0}^n c_j^2\right)^{1/2}=
%\left( \sum_j \left(\frac{L_j}{4 \pi^2 k^2} - \frac{1}{8\pi^4 k^4 L_j}\left( 1- \cos(2 \pi k L_j)\right)\right) \right)^{1/2}=
\frac 1{2 \pi |k|}\left(  1 - \frac{1}{2 \pi^2 k^2}\sum_{j=0}^n\frac{1- \cos(2 \pi k L_j)}{L_j}\right) ^{1/2}
=\frac 1{2 \pi |k|}\left(  1 - \frac{1}{ \pi^2 k^2}\sum_{j=0}^n\frac{\sin^2(\pi k L_j)}{L_j}\right) ^{1/2} .
\]

To make the worst case error as small as possible, we want to find the optimal distribution
of information nodes $(x_j)_{j=1}^n$, in  particular for large $n$. That is,
    \[
    \inf\limits_{L_j\ge0,\atop \sum\limits_{j=0}^n L_j =1}
   \frac 1{2 \pi |k|}\left(  1 - \frac{1}{ \pi^2 k^2}\sum_{j=0}^n\frac{\sin^2(\pi k L_j)}{L_j}\right) ^{1/2}.
   \]

For this, we prove the following lemma.

\begin{lem}\label{lem:2-7}
Let $k\in\Z\backslash\{0\},\ 0=x_0< x_1 < x_2 < \dots < x_n < x_{n+1}=1$ and $L_j= x_{j+1}-x_j,\ j=0,1,\ldots,n$.
Suppose that $n+1 \ge 2.7 |k|$.  Then

\begin{equation}\label{eq:optim1}
    \sup\limits_{L_j\ge0,\atop \sum\limits_{j=0}^n L_j =1}
   \sum_{j=0}^n \frac{ \sin^2(\pi k L_j)}{L_j} = (n+1)^2 \sin^2\left(\frac{\pi k}{n+1}\right),
\end{equation}
i.e., equidistant $x_j$ with $L_j = \frac 1{n+1}$ for all $j=0,1,\ldots,n$ are  optimal.
\end{lem}

\begin{proof}
Let $f(x)= \sin^2(\pi k x)/x,\ x\in(0,1]$, and $k\in \N$ without loss of generality. Then we have
\[
f^\prime(x)=\frac{1}{x^2} \left(  \pi k x \sin(2\pi k x)
- \sin^2(\pi k x)\right)=\frac{\sin(\pi k x)}{x^2} \left( 2 \pi k x \cos(\pi k x) - \sin(\pi k x)\right).
\]

Solving the equation, $2t \cos(t) = \sin(t)$, i.e., $\tan(t)=2t$,
on the interval $(0,k \pi]$ with $t=\pi k x$, we get $k$ solutions, $t^\ast_0,\ t^\ast_1, \ldots, t^\ast_{k-1}$, with
$j \pi + \pi/3 < t^\ast_j <  j \pi + \pi/2$ for $j=0,1, \ldots, k-1$ and
$t^\ast_{j+1} - (t^\ast_j + \pi)>0$ for $j=0,1, \ldots, k-2$.
This implies that
\[
\sin(2 t^\ast_j)> \sin(2 t^\ast_{j+1})>0,\quad j=0,1, \ldots, k-2.
\]
In particular, $t^\ast_0\approx 0.3710 \pi
%S Aug.11  delete "\approx 1.1656 "
< \pi/2$.
A figure of the function $f(x)=\sin^2(\pi k x)/x$ on $(0,1]$
with $k=6$ is drawn by using Matlab, see Figure \ref{fig:1}.

%E I guess that one figure (k=6) is enough.

\begin{figure}[htb]
\begin{center}
\scalebox{.68}{\includegraphics[width=19cm]{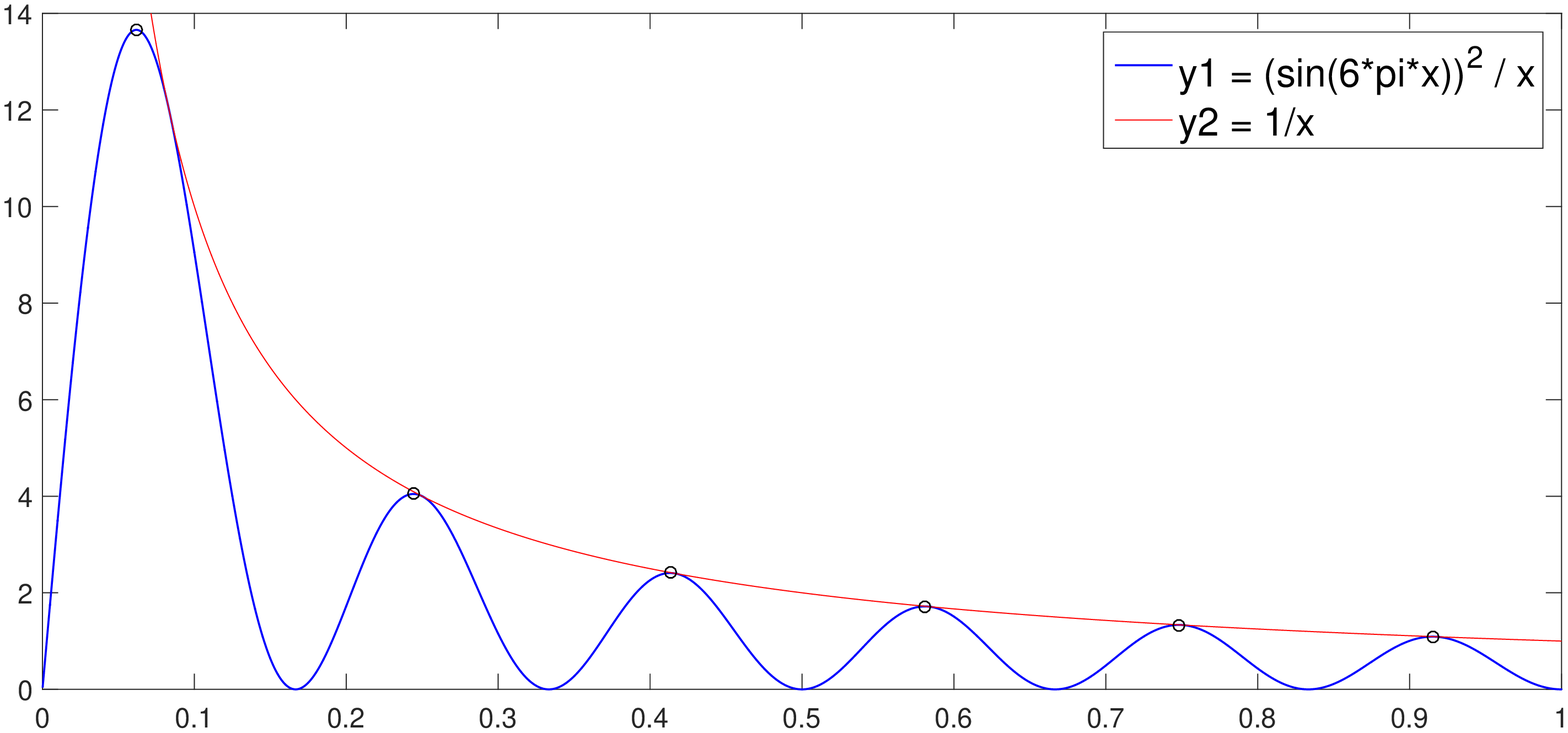}}\\
% size-20-14
% \scalebox{.68}{\includegraphics[width=19cm]{figure-724-k8.eps}}\\
% \scalebox{.78}{\includegraphics[width=19cm]{figure-black-722-k6.eps}}\\
% size-20-14
% \scalebox{.68}{\includegraphics[width=19cm]{figure-red-722-k8.eps}}\\
% size-20-14
% \scalebox{.68}{\includegraphics[width=19cm]{figure0722-k6.eps}}\\
% %size-22-16-14
% \scalebox{.68}{\includegraphics[width=19cm]{figure-black-20.eps}}\\
% size-20-16-14
% \scalebox{.68}{\includegraphics[width=19cm]{figure-black-20-1.eps}}\\
%  size-20-16-14
% \scalebox{.68}{\includegraphics[width=19cm]{figure-black-20-2.eps}}\\
% % size-20-16-14
\caption{
%S Fig. 1.~~
$y_1=\sin^2(\pi k x)/x$ on $(0,1]$
for $k=6$,\; in contrast to $y_2= 1/x$}\label{fig:1}
\end{center}
\end{figure}

The function $f$ is nonnegative and $f(j/k)=0$ for all $j=1, \ldots, k$.
%S  Aug.11                             ^^^ corrected
We also put $f(0)=0$ for continuity. In each interval $[j/k,(j+1)/k]$ with $j=0,1, \ldots, k-1$,
the point $x_j^\ast:=t^\ast_j/(k \pi)$ is the maximum point of the function $f$. Since
\[
f(x_j^\ast)=f\left(\frac{t_j^\ast}{k \pi}\right)
= k \pi \frac{\sin^2(t_j^\ast)}{t_j^\ast} = 2 k \pi \frac{\sin^2(t_j^\ast)}{\tan(t_j^\ast)} =  k \pi \sin(2 t_j^\ast),
\]
one knows that $f(x_0^\ast)$ is the maximum value on the whole interval
$[0,1]$. Moreover, the function $f$ is monotone increasing on $[0,x_0^\ast]$ with
%S Aug.11,    instead of "strictly"    ^^^^^^^              ^^ "the interval" deleted
 $x_0^\ast\approx 0.3710/k < 1/(2k)$.

Next, for the second derivative of $f$ on the interval $(0,x_0^\ast]$, we have
\[
f^{\prime\prime}(x) = \frac 2{x^3} \left( \sin^2(\pi k x) -  \pi k x \sin(2 \pi k x) +  \pi^2 k^2 x^2\cos(2 \pi k x)\right).
\]
For this, we take $G(t)=\sin^2(t) -  t \sin(2 t) +  t^2\cos(2 t)$ with $t=\pi k x $.
%S Aug. 12         ^^^^^^ take out the multiple number 2 for simplicity.
Then $G^\prime(t)= -2 t^2 \sin(2 t)< 0$ and $G(t) < G(0) = 0$ for $t\in (0,\pi/2)$,
which yields that $f^{\prime\prime}(x)< 0 = f^{\prime\prime}(0)$ holds true on the
whole interval $(0,x_0^\ast]$. Indeed, $f^\prime(0) = \pi^2 k^2$, and using L'H$\hat{\rm o}$pital's rule,
\[
f^{\prime\prime}(0) = \lim_{x\rightarrow 0^+} f^{\prime\prime}(x)= 2 (\pi k)^3  \cdot \lim_{t\rightarrow 0^+}
\frac{G^{\prime}(t)}{(t^3)^{\prime}}=0.
\]

Now we turn to the distribution of the nodes $x_j, j=1,\ldots,n$,
%S  Aug. 11                           ^^^^^ this was "points"
 with $n + 1\ge 1/x_0^\ast \approx 2.6954 k$.
We consider two cases depending on whether $L_{j}> x_0^\ast$ holds for some $j\in \{0,1,\ldots, n\}$ or not.
We first assume that  $L_j\le x_0^\ast$ for all $j=0,\ldots,n$.
Thanks to the above properties on the first and second derivatives of $f$, we know that $f$ is concave on $(0,x_0^\ast]$.
Using the Lagrange multiplier method, we obtain that equidistance is the optimal case, i.e., $L_0 = L_1 = \cdots = L_{n}$.
If $L_{j_0}> x_0^\ast$ holds for some $j_0\in \{0,1,\ldots, n\}$, we have $f(L_{j_0})<f(x_0^\ast)$
and easily construct a better distribution, $\{L^{(1)}_j\}_j\subset(0,x_0^\ast]$, in the following steps.

Step 1: Since $L_{j_0}> x_0^\ast$ and $n+1 \ge \frac 1{x_0^\ast} \approx 2.7 k$,
we define $L^{(1)}_{j_0}=x_0^\ast$ simply, which ``saves" $L_{j_0} - x_0^\ast$ (on length) for the summation to compute.

Step 2: For the other $j$ satisfying $L_{j}> x_0^\ast$, we repeat Step 1.

Step 3: Due to $n+1 \ge \frac 1{x_0^\ast}$ and $L_{j_0}> x_0^\ast$, there exists some $j$
satisfying $L_{j}< x_0^\ast$. The ``saving" length of $L^{(1)}_{j}$'s from
Steps 1 and 2 can be given to those $L_j< x_0^\ast$, such that for all $j=0,1\ldots,n$,
$L_j \le L^{(1)}_{j}\le x_0^\ast$ and $\sum\limits_j L^{(1)}_{j} = 1$.

Since the function $f$ is          monotone increasing on $[0,x_0^\ast]$,
%S Aug.11,  instead of "strictly"  ^^^^^^^               ^^ "the interval" deleted
the above steps yield that $ f(L^{(1)}_{j}) \ge f(L_j)$ holds true for all  $j=0,1\ldots,n$.
Hence, the proof is finished as required.
\end{proof}

%S(22 July)   We can skip it. Indeed we will use this remark
% (with x_1 and x_n given and x_n-x_1<1)
%  to prove x_{j+1}-x_j = (x_n-x_1)/(n-1) in Lemma {lem:nonT-27-nonBVP}
%  and then the final result Thm. 13 (last Thm. for the full H^1 in Section 3.2).

%  \begin{rem}\label{re:average-L}
%  In Lemma \ref{lem:2-7}, the condition, $\sum_j L_j =1$, can be loosen and
%  replaced by, the average value, $\frac 1{n+1}\sum_{j=0}^n L_j \le x_0^\ast
%  \approx 0.3710/|k|$, or $\le 1/(2.7 |k|)$. Then equidistant $x_j$ with $L_0=\ldots=L_n$  are optimal.
%  \end{rem}

\begin{rem}\label{re:k-R}
In the case $k\in\R\backslash\{0\}$, Lemma \ref{lem:2-7} stays the same.
There are only some small modifications in the proof. Firstly,
%S 22 July: For the function f defined in the proof of the last Lemma \ref{lem:2-7}
%   $f(1)\neq 0$  if $k\notin\Z$, which is not important.
for the equation, $2t \cos(t) = \sin(t)$,
on the interval $(0,|k| \pi]$ with $t= |k| \pi x$,
the number of solutions is $\lfloor |k| \rfloor$  or $\lceil |k| \rceil$, depending on $|k|$.
%S(03 July) : where $\lfloor a \rfloor$ denotes the largest integer
%no exceeding $a\in \R$ and $\lceil a \rceil$ denotes the smallest integer
% no smaller than $a\in \R$.
Secondly, the expressions of $f^\prime$ and $f^{\prime\prime}$ remain the same, respectively.
In the other expressions for positive numbers, one can replace $k$ by $|k|$ simply.
%S Aug.11,                                     ^^^ instead of "we"
In particular, the maximum point of $f$ is $x_0^\ast\approx 0.3710/|k|$
          in $[0,1]$ if $|k|\ge 0.3710$, otherwise $1$.
%S Aug.11 ^^^^^^^^^^^^^^^^^^^^^^^^^^^^^^^^^^^^^^^^^^^^^^^ for general k\in\R
\qed
\end{rem}

We are now ready to give sharp estimates on the worst case error.

\begin{thm}\label{thm3}
In the case of $H^1_0([0,1])$ with $\rho_k(x) = \exp(- 2 \pi i k x)$ and $k\in\R\backslash\{0\}$,
%S (03 July)                       ^^^^^^^^^^^^^^^^^^^^^^^^^^^^^^^^^^^^ added
the radius of information is given by
\[
r(N) = \frac 1{2 \pi |k|}\left(  1 - \frac{1}{ k^2 \pi^2}\sum_j\frac{\sin^2(\pi k L_j)}{L_j}\right) ^{1/2},
\]
where $L_j= x_{j+1}-x_j,\ j=0,1,\ldots,n$, with $0=x_0< x_1 < x_2 < \dots < x_n < x_{n+1}=1$.

Moreover,    if $n \ge 2.7 |k|-1$, then equidistant nodes
($L_j = \frac 1{n+1},\ j= 0, 1,\ldots,n$) are optimal and the worst case error is
\[
e (n, I_{\rho_k}, H^1_0) = \frac1{2\pi |k|} \left( 1 - \frac{(n+1)^2}{k^2\pi^2}
\sin^2\left(\frac{ k \pi}{n+1} \right) \right)^{1/2}.
\]
\qed
\end{thm}

%^^^^^^^^^^^^^^^^^^^^^^^^^^^^^^^^^^^^^^^^^^^^^^^^^^^^^^^
%S  August 11: add two points below.

Furthermore, we establish a few nice asymptotic properties for the $n$th minimal errors as follows.

\begin{cor}\label{cor: asym-2sq3}
Under the same assumption of Theorem~\ref{thm3}, the following statements
hold:

\begin{enumerate}[(i)]
\item
For fixed $k\in \R\backslash\{0\}$ and (optimal) equidistant nodes, we have
%S 22 July:                         ^^^^^^^^^^^^^^^^^^^^^^^^^^^^^^ added
$$
\lim_{n \to \infty}  e (n, I_{\rho_k}, H^1_0) \cdot n =
\frac{1}{2 \sqrt 3} .
$$

\item For fixed $n\in \N$ and arbitrary nodes, we have
$$
\lim_{|k| \to \infty}  e (n, I_{\rho_k}, H^1_0) \cdot |k| =
\frac{1}{2 \pi} .
$$
\item Suppose in addition that $k\in \Z\backslash\{0\}$. Then
for fixed $n\in \N$ and arbitrary nodes,
$$
\lim_{|k| \to \infty}  \frac {e (n, I_{\rho_k}, H^1_0)}
{e (0, I_{\rho_k}, H^1_0)} =1.
$$
\end{enumerate}
\qed
\end{cor}

\begin{proof}
Point (i) can be proved via Taylor's expansion in the same manner as in Theorem~\ref{thm6}. Point (ii)  is known from the result for the radius of information in Theorem~\ref{thm3}. This implies point (iii) by Proposition~\ref{prop2}.

\end{proof}
\begin{rem}\label{re:small-n}
What is the optimal distribution of information nodes if $n+1< 2.6954 |k|$
for the supreme term on the left side of  \eqref{eq:optim1}?
In the case of $n=|k|-1$ equidistant nodes are the worst nodes.
Observe that in this case these $n$ function values are useless:
the radius of information is the same as the initial error of the problem.
%S Aug. 31 The following sentence is added due to the comment from Bedlewo.
Further, 
the radius of information on equidistant nodes is oscillatory (no more than the initial error)
as $n$ increases from $1$ to $\lfloor |k| \rfloor -1$ if $|k|\ge 3$, and it is monotone decreasing (with the asymptotic constant $\frac{1}{2 \sqrt 3}$ mentioned above) as $n$ increases from $\max(1,\lceil |k| \rceil-1)$ to infinity.

Even, in the cases $n+1=2|k|, 2.5 |k|, 2.6 |k|$  we can show that
equidistant nodes are not always optimal by Matlab experiments,
see Table \ref{tab:example1}.
We compare the worst case errors
$\hat e^{\rm equi}_n :=  e^{\rm equi} (n, I_{\rho_k}, H^1_0)$
(for nodes $x_j=\frac{j}{n+1}$)
 %S(22 July)           ^^^^^^^^^^^^^^^^^^^ added to avoid misunderstanding.
 with $\hat e^{\rm opt}_n := e^{\rm opt} (n, I_{\rho_k}, H^1_0)$ (for optimal nodes)
and compute the so-called relative errors, $\hat d^{\rm equi}_n := (\hat e^{\rm equi}_n - \hat e^{\rm opt}_n)/ \hat e^{\rm opt}_n$.

\vspace{3mm}
\begin{table}[!htbp]
\footnotesize
\caption{Counterexamples for equidistance by Matlab}\label{tab:example1}
\centering
\renewcommand\arraystretch{2}
\begin{tabular}{rrrrrrr}
\hline
\ $\frac{n+1}k$ & \quad\quad$k$ & \quad$n+1$ & $\hat e^{\rm equi}_n$ & $\hat e^{\rm opt}_n$  &
$\hat e^{\rm equi}_n - \hat e^{\rm opt}_n$ & $\hat d^{\rm equi}_n$
\\
\hline
 %2 & 7 & 14 &  $1.75338 \cdot 10^{-2}$ & $1.71236 \cdot 10^{-2}$ &
 %$4.10 \cdot 10^{-4}$ & +2.4\% \\
 %\hline
 2 & 72 & 144 &  $1.68133 \cdot 10^{-3}$ & $1.60478 \cdot 10^{-3}$ &
 $7.66 \cdot 10^{-5}$ & +4.8\% \\
 %\hline

% 2 & 290 & 580 &  $4.11868 \cdot 10^{-4}$ & $3.92285 \cdot 10^{-4}$ &
% $1.96 \cdot 10^{-5}$ & +5.0\% \\
% %\hline

% 2.5 & 66 &  165 & $1.57615 \cdot 10^{-3}$ & $1.57503 \cdot 10^{-3}$ &
% $1.12 \cdot 10^{-6}$ & +0.071\% \\
% %\hline

%2.5 & 104 &  260 & $1.00025 \cdot 10^{-3}$ & $9.98215 \cdot 10^{-4}$ &
%$2.04 \cdot 10^{-6}$ & +0.20\% \\
%\hline
2.5 & 194 &  485 & $5.36217 \cdot 10^{-4}$ & $5.34544 \cdot 10^{-4}$ &
$1.67 \cdot 10^{-6}$ & +0.31\% \\
%\hline

%2.5 & 396 &  990 & $2.61372 \cdot 10^{-4}$ & $2.60393 \cdot 10^{-4}$ &
 %$9.79 \cdot 10^{-7}$ & +0.38\% \\
% %\hline

2.6 & 290 & 754 &~\quad $3.47616  \cdot 10^{-4}$ & ~\quad $3.47567
\cdot 10^{-4}$ &\quad\quad $4.90  \cdot 10^{-8}$ & ~\quad +0.014\% \\
\hline
\end{tabular}
\end{table}

Related to the field of digital signal processing, a famous assertion,
the Nyquist Sampling Theorem states that, see \cite{Lan67,ME09}:
%S (22 July) see originally from  http://www.spot.pcc.edu/~ghecht/Gary%27s%20Nyquist%20Document.pdf
If a time-varying signal is periodically sampled at a rate of
at least \emph{twice} the frequency of the highest-frequency sinusoidal component contained
%S Aug.11 ^^^^ stress
within the signal, then the original time-varying signal can be exactly recovered from the periodic samples.
%S (22 July) I improved the following statements.
It seeks in essence for the reconstruction of continuous periodic functions.
In contrast, for oscillatory integrals of periodic
functions from $H^1$, the multiple number $2.7$ assures that equidistant nodes achieve the optimal quadrature.
\qed
\end{rem}

%E Should we skip the next remark for simplicity?
%S  (22 July) Sure.

%  \begin{rem}\label{re:osci-period}
%  For the oscillatory integrals, one can also consider $\rho_k(x) = \exp(- 2 \pi i k x)$ with $k\in\R\backslash\{0\}$.
%  % In addition, suppose $k\in\Z\backslash\{0\}$.
%  % Then Theorem~\ref{thm3} remains the same for the space $\wt H^1([0,1])$ with boundary values.
%  In this case Theorem~\ref{thm3} stays the same for the space $\wt H^1([0,1])$ with boundary values.
%  Still, $A_0(f):= I_{\rho_k}(f(0)) = f(0) \cdot I(\rho_k)$ and $I_{\rho_k}(f) - A_0(f)
%  = I_{\rho_k}(f-f(0))$ with $f-f(0)\in H^1_0$. This helps us to shift the problem from
%  the space $\wt H^1$ to the space $H_0^1$ as before. As is calculated in Prop. \ref{prop2}, the initial error is
%  \[
%  \begin{split}
%  e_0^{\rm b}\left( \wt H^1, \rho_k \right) &:= \sup_{\Vert f \Vert_{\wt H^1}
%  \le 1 } | I_{\rho_k} (f) - A_0(f) |  = \sup_{\Vert f \Vert_{\wt H^1} \le 1 } | I_{\rho_k} (f-f(0)) |\\
%  &  = \sup_{\Vert f \Vert_{ H^1_0} \le 1 } | I_{\rho_k} (f) |
%  = \Vert R-c \Vert_2  = \frac 1{2 \pi \abs{k}} \left( 1- \frac{\sin^2(\pi k)}{\pi^2 k^2} \right)^{1/2},
%  \end{split}
%  \]
%  and the formula of the worst case function is complicated and omitted here.
%  \end{rem}

\subsection{The general case}   \label{Se:non-period}

%  {\bf E:}  This subsection is a bit long. I like the final result,
%  Theorem~\ref{thm5}, very much!!!
%  But what about Theorem 12 (last version). Should we skip it?
%  I guess that it is better to have a direct proof of Theorem~\ref{thm5} and to
%  have few but nice results. What do you think?
%
%  {\bf S:} Sure. I agree with you completely. I removed the case with boundary value $f(0)$ and
%  rewrote this subsection, mainly the proof of Lemma~\ref{lem:nonT-27-nonBVP}. At the end of Sect. \ref{Se:osci4} or Intro.,
%  I suggest to give a short remark for the boundary problem related to Thm.~12 (last version) if convenience.

We want to find optimal nodes,
\[
0 \le  x_1 < \cdots < x_n \le 1,
\]
for the oscillatory integrals and integrands from the full space
$H^1 ([0,1])$ with $k\in \R\backslash\{0\}$.
%S Aug.10,              ^^^^        instead of "\Z"
We will prove some nice formulas for
large $n$, but not for small $n$. For convenience we take $x_0 = 0,\ I_j = [x_j,x_{j+1}]$ and $L_j=|I_j|$.

To compute the number $r(N)$ as in Theorem~\ref{thm2} with arbitrary nodes mentioned above, firstly  we consider the initial errors for all intervals
               under the assumption $N(x_1, \dots, x_n)= 0$.
%S Aug. 11:    ^^^^^ this was "with".

On the intervals $I_j,\ j=1,\ldots,n-1$, we know from \eqref{eq:initial-R-c} that
the initial error is $\norm{R-c}_{2,j}$ with
\[
\norm{R-c}_{2,j}^2 :=\int_{x_j}^{x_{j+1}} \big( R(x)-c \big)
\big( \overline{R(x)}-\overline{c} \big)  \dx=\frac{L_j}{4 \pi^2 k^2} - \frac{1}{8\pi^4 k^4 L_j}\left( 1- \cos(2\pi k L_j)\right).
\]

On the interval $I_0=[0, x_1]$, we obtain from  \eqref{eq:initial-R} that
the initial error is $\norm{R_0}_{2,0}$ with
$R_0(t)=\int_0^t \rho_k(x) \dx$,\ $t\in[0, x_1]$,
and for $k\in \R\backslash\{0\}$,
%S            ^^ instead of \Z

 \[
\norm{R_0}_{2,0}^2 :=\int_{0}^{x_1} R_0(x) \cdot \overline{R_0(x)}
\dx=\frac{1}{4 \pi^2 k^2} \left( 2 L_0 -  \frac{\sin(2 \pi k L_0)}{\pi k}\right).
\]

Similarly on the interval $I_n=[x_n, 1]$, the initial error is
$\norm{R_n}_{2,n}$ with $R_n(t)=\int_t^1 \rho_k(x) \dx$,\ $t\in[x_n,1]$, and for $k\in\R\backslash\{0\}$,
%S Aug.10,  improved  ^^^^^^^^^^^^^^^^^^^^^^^^^^^^^^^^^^^^^^^^^^^^^^^^^^^^^^^^^^^^^^^^^^^^^^^^
\[
\norm{R_n}_{2,n}^2 :=\int_{x_n}^{1} R_n(x) \cdot \overline{R_n(x)}  \dx
=\frac{1}{4 \pi^2 k^2} \left( 2 L_n -  \frac{\sin(2 \pi k L_n)}{\pi k}\right).
\]
%S Aug.10, I have to show more details here by comments.

\begin{comment}
\begin{equation*}
R_n(x):= \int_x^1 \rho_k(t) {\rm d} t = \int_x^1 e^{- 2\pi i k t} {\rm d} t = \frac{e^{- 2\pi i k }- e^{- 2\pi i k x}}{- 2\pi ik}= \frac{e^{- 2\pi i k }(1- e^{ 2\pi i k (1-x)})}{- 2\pi ik},
\end{equation*}
Then on the interval $[x_n,1]$,
\begin{equation*}
\norm{R_n}_{2,n}^2
=  \int_{x_n}^{1}  \frac{e^{- 2\pi i k }- e^{- 2\pi i k x}}{-2\pi ik} \frac{e^{ 2\pi i k }- e^{ 2\pi i k x}}{2\pi ik} \dx
= \int_{0}^{1-x_n} \frac{2 \left(  1 - \cos(2 \pi k t) \right)}{4 \pi^2 k^2} {\rm d} t
= \frac{2 L_n -  \frac{\sin(2 \pi k L_n)}{\pi k}}{4 \pi^2 k^2}.
\end{equation*}
I misunderstood that $\sin(2 \pi k)=0$ should be satisfied.
\end{comment}

As usual, the initial error is given by taking the zero algorithm $A_0(f)=0$.
If $k\in\Z\backslash\{0\}$, we have  $I(\rho_k)=0$, and
%S Aug. 11 (rewrite)  "Since $I(\rho_k)=0$ for $k\in\Z\backslash\{0\}$, we have,"
by \eqref{eq:initial-R-b},
\[
e(0, I_{\rho_k}, H^1) =
\sup_{f\in H^1:\, \|f\|\le1}\, |I_{\rho_k} (f)| = \sup_{f\in H^1:\, \|f\| \le 1}\, |I_{\rho_k} (f-f(0))|
% = e_0\left( H^1,\rho_k \right)
= \frac{\sqrt{2}} {2 \pi |k|}.
\]
Following the same lines as in Section \ref{Se:period}, the radius of information is,
\[
\begin{split}
&\left(\sum_{j=1}^{n-1} \norm{R-c}_{2,j}^2 + \norm{R_0}_{2,0}^2 + \norm{R_n}_{2,n}^2
%S            ^^ instead of 0
\right)^{1/2}\\
&\quad\quad\quad\quad=
\frac 1{2 \pi |k|}\left(  L_0 -  \frac{\sin(2 \pi k L_0)}{\pi k} + L_n -  \frac{\sin(2 \pi k L_n)}{\pi k} + 1
- \frac{1}{ \pi^2 k^2}\sum_{j=1}^{n-1}\frac{\sin^2(\pi k L_j)}{L_j}\right) ^{1/2} .
\end{split}
\]
Suppose in addition that
$n-1 \ge 2.7 |k|$. Following Lemma \ref{lem:2-7},
we obtain that for any fixed nodes $ x_1, x_n \in [0,1]$, equidistant $x_j$ with $L_{j-1} =
\frac {x_n-x_1}{n-1}$ for all $j= 2,\ldots,n-1$ are  optimal. Afterwards, we have to find the optimal nodes, $x_1=:x$ and $x_n=: 1-z$, for
\[
\inf\limits_{x,y,z\ge0,\atop x + (n-1)y +z =1}
\frac 1{2 \pi |k|}\left(  x -  \frac{\sin(2 \pi k x)}{\pi k} +  z -
\frac{\sin(2 \pi k z)}{\pi k} + 1 - \frac{(n-1)}{ \pi^2 k^2}\frac{\sin^2(\pi k y)}{y}\right) ^{1/2}.
\]
For this, we prove the following lemma.

\begin{lem}\label{lem:nonT-27-nonBVP}
Let  $k\in \R \backslash\{0\}$ and  $n-1 \ge 2.7 |k|$.  Then for
%S Aug.10, ^^ improved from \Z
\begin{equation}\label{eq:nonT-27-nonBVP}
    \inf\limits_{x,y,z\ge0,\atop x + (n-1)y +z =1}
%S Aug. 29 \frac 1{2 \pi |k|}   deleted for simplicity
\left(  x -  \frac{\sin(2 \pi k x)}{\pi k} +  z -
\frac{\sin(2 \pi k z)}{\pi k} + 1 - \frac{(n-1)}{ \pi^2 k^2}\frac{\sin^2(\pi k y)}{y}\right) ^{1/2},
\end{equation}
the unique solution of the minimum point $(x,z)=(x^\ast, z^\ast)$ satisfies
that $x^\ast= z^\ast$ and $x^\ast$ is the stationary point of the function,
\[
S(x)= 2x -  \frac{2\sin(2 \pi k x)}{\pi k}  - \frac{(n-1)^2}{ \pi^2 k^2}
\frac{\sin^2\left(\pi k \cdot \frac{1-2x}{n-1}\right)}{1-2x},
%\quad\quad x\in\left(0,\ \frac1{6 |k|}\right),
\]
in the interval $\left(0,\  \min\left(\frac12,\frac1{6|k|}\right)\right)$ and dependent of $|k|$ and $n$.
\qed
\end{lem}

\begin{proof}
Without loss of generality, we assume that $k\in \R^+$ and $z\le 1/2$
%S Aug.10,                                       ^^^^ improved from \N
since $x+(n-1)y+z=1$. Let us begin with the steps below.
%S Aug.10,                   ^^^^ this was "prove".

Step 1: For any fixed $z=z_0$, we discuss by $x+(n-1)y=1-z_0\ge 1/2$.
%S Aug. 11             ^^ added
We prove now that the unique solution
of the minimum point for $x=\wt x$ (depending on $k,n$ and $z_0$) should appear in $(0,1/(6k))$.

That is to consider
\[
    \inf\limits_{x,y\ge0,\atop x + (n-1)y =1-z_0}
   \left( x -  \frac{\sin(2 \pi k x)}{\pi k}  - \frac{(n-1)}{ \pi^2 k^2}\frac{\sin^2(\pi k y)}{y}\right).
   \]
We take $f(y)= \sin^2(\pi k y)/y,\ y\in(0,1],\ f(0)=0$ as in Lemma~\ref{lem:2-7},
\[
F_1(x) = \frac{(n-1)}{ \pi^2 k^2} f(y) \quad \text{with} \quad y=\frac{1-z_0-x}{n-1}
\in \left[0,\frac{1-z_0}{n-1}\right],
 \]
and
\[
g_1(x)=x - \frac{\sin(2 \pi k x)}{\pi k},\ \ x\in[0,1-z_0].
\]
%  Then
% $$g_1^{\prime}(x)= 1 - 2  \cos(2 \pi k x)<0,\quad x\in\left[0,\frac1{6k}\right),  \ \ g_1^{\prime}(\frac1{6k})=0.$$

Define $S_1(x):= g_1(x)- F_1(x),\ x\in[0,1-z_0]$.  Under the assumption of $n-1 \ge 2.7 |k|$, we have $0 < \frac{1-z_0}{n-1} \le \frac{1}{n-1}
< x_0^\ast \approx 0.3710/k$. Afterwards, for $y\in[0,\frac{1-z_0}{n-1}],\ f^\prime(y)>0$ and
\[
F_1^\prime(x) = \frac{(n-1)}{ \pi^2 k^2}\cdot y^\prime_x\cdot f^\prime(y)
= -\frac{1}{ \pi^2 k^2} f^\prime(y)<0,\ x\in\left[0,1-z_0\right].
\]
This helps us to decompose the function $S_1(x)$ into two parts, $x-F_1(x)$ and $-\sin(2 \pi k x)/(\pi k)$.
One part, $x-F_1(x)$, is monotone increasing on $[0,1-z_0]$. The other, $-\sin(2 \pi k x)/(\pi k)$, is $1/k$ -periodic.
This implies that the minimum point $x=\wt x$ for $S_1(x)$ appears in $[0,1/(4k)]$, more precisely  $(0,1/(6k))$, see the details below.

% ^^^^^^^^^^^^^^^^^^^^^^^^^^^^^^^^^^^^
%S Aug. 14, improved for the general case k\in\R, instead of \Z
Indeed, one can assume first that $1/(6k)\le 1-z_0$. Then $g_1^\prime(x)=1-2\cos(2\pi k x)$ is increasing from $-1$ to $0$ on the interval $[0,1/(6k)]$ and
positive on $(1/(6k),1/(4k)]$.  Since $f^{\prime\prime}$ is negative on $(0,x_0^\ast]$ and $f^{\prime}(x_0^\ast)=0$, we obtain that
\[
F_1^{\prime\prime}(x) = -\frac{1}{ \pi^2 k^2} \cdot y^\prime_x\cdot f^{\prime\prime}(y)=
\frac{1}{ \pi^2 k^2}\cdot \frac 1{n-1} \cdot f^{\prime\prime}(y) < 0,\quad x\in[0,1-z_0],
\]
and for some $\delta_1 \in (0,1)$,
\[
F_1^\prime(0) =  -\frac{1}{ \pi^2 k^2} f^\prime\left(\frac {1-z_0}{n-1} \right)=:
-\delta_1  >  F_1^\prime\left(\frac1{6k}\right) \ge
F_1^\prime(1-z_0) =  -\frac{1}{ \pi^2 k^2} f^\prime(0)=-1.
\]
By the intermediate value theorem, there is one point $\wt x\in (0,1/(6k))$ such that $S_1^\prime(\wt x)
=g_1^\prime(\wt x)-F_1^\prime(\wt x)=0$. Moreover, the monotonicity of $g_1^\prime - F_1^\prime$ assures the uniqueness of $\wt x$.

% ^^^^^^^^^^^^^^^^^^^^^^^^^^^^^^^^^^^^
In the case $1/(6k)> 1-z_0$, we have that $g_1^\prime(x)=1-2\cos(2\pi k x)$ is increasing from $-1$ to $-\delta_2$ on the interval $[0,1-z_0]$ where  $\delta_2 \in (0,1)$.  Similarly,
 for some $\delta_1 \in (0,1)$,
\[
F_1^\prime(0) =  -\frac{1}{ \pi^2 k^2} f^\prime\left(\frac {1-z_0}{n-1} \right)=:
-\delta_1  >
F_1^\prime(1-z_0) =  -\frac{1}{ \pi^2 k^2} f^\prime(0)=-1.
\]
Then there is one point $\wt x \in (0,1-z_0) \subset (0,1/(6k))$ such that $S_1^\prime(\wt x)
=g_1^\prime(\wt x)-F_1^\prime(\wt x)=0$, and the monotonicity of $g_1^\prime - F_1^\prime$ assures the uniqueness of $\wt x
\in \left(0,\min\left(1-z_0,\frac1{6k}\right)\right)$.
%S Aug.11 ^^^^^^^^^^^^^^^^^^^^^^^^^^^^^^^^^^^^

Step 2: Iterate the above process by fixing $\wt x$. From $z+(n-1)y=1-\wt x> 1-\min\big(1-z_0,\frac1{6k}\big) \ge z_0$,
we obtain a new minimum point for $z=\wt z \in \left(0,\min\left(1-\wt x,\frac1{6k}\right)\right)$.

Step 3: Iterate the process by fixing $y$ above.
One knows easily
%S Aug.11 ^^^^^^
$x+z=1-(n-1)y<1/(3k)$ and considers
$$g_1(x)=x - \frac{\sin(2 \pi k x)}{\pi k}, \quad x\in\left[0,\frac1{3k}\right].$$
Then
$$g_1^{\prime\prime}(x)= 4 \pi k \sin(2 \pi k x)>0,\quad x\in\left(0,\frac1{3k}\right].$$
This implies, by Lagrange multiplier method, that $x=z=p/2$ is the unique solution of the extremal problem,
\[
    \inf\limits_{x,z\ge0,\atop x + z = p }
\left(  x -  \frac{\sin(2 \pi k x)}{\pi k} +  z -
\frac{\sin(2 \pi k z)}{\pi k} \right) \quad \quad
\text{for any fixed}\quad   p \in \left[0, \frac 1{3k}\right].
%S Aug.29                  ^ this was "0 < p \le \frac 1{3k}"
\]

The above three steps shift the extremal problem
\eqref{eq:nonT-27-nonBVP} to the simpler case below,
%S Aug.29                    ^^^ this was "a"
\[
\inf\limits_{x,y\ge0,\atop 2x + (n-1)y =1}
%S Aug.29                  ^ typo
\left( 2x -  \frac{2\sin(2 \pi k x)}{\pi k}  - \frac{(n-1)}{ \pi^2 k^2}\frac{\sin^2(\pi k y)}{y}\right).
\]

We follow Step 1 with a few small modifications mainly on domains. Here we take
\[
F(x) = \frac{(n-1)}{ \pi^2 k^2} f(y) \quad \text{with} \quad y=\frac{1-2x}{n-1}
\in \left[0,\frac{1}{n-1}\right],
 \]
\[
g(x)=2 g_1(x)=2x - 2\frac{\sin(2 \pi k x)}{\pi k},\ \ x\in\left[0,\frac12\right],
\]
and $S(x):= g(x)- F(x),\ x\in[0,\frac 12]$.
%S Aug.29                       ^^^^^^^^ typo
Since $n-1 \ge 2.7 |k|$, we have $y\le \frac 1{n-1}\le \frac 1{2.7|k|}$ and
%S Aug.29                        ^^^^^^^^^^^^^^^^^^^^^^^^^^^^^^^^^^^^^^^^^^ added
$F^\prime(x) <0,\ x\in\left[0,\frac12\right]$.

This helps us to decompose the function $S(x)$ into two parts again, and the minimum point for $S(x)$ appears in  $(0,1/(6k))$.

To be specific, $g^\prime(x)=2-4\cos(2\pi k x)$ is increasing from $-2$ to $0$ on
 $[0,1/(6k)]$. Since $f^{\prime\prime}$ is negative on $(0,x_0^\ast]$ and $f^{\prime}(x_0^\ast)=0$, again we have
\[
F^{\prime\prime}(x) = -\frac{2}{ \pi^2 k^2} \cdot y^\prime_x\cdot f^{\prime\prime}(y)=
\frac{4}{ \pi^2 k^2}\cdot \frac {1}{(n-1)} \cdot f^{\prime\prime}(y) < 0,\ x\in
\left[0,\frac12\right).
%S Aug. 29     ^^^^ typo
\]
% ^^^^^^^^^^^^^^^^^^^^^^^^^^^^^^^^^^^^
%S Aug. 14, improved for the general case k\in\R, instead of \Z
If $1/(6k)\le 1/2$, then for some $\delta \in (0,1)$,
\[
F^\prime(0) =  -\frac{2}{ \pi^2 k^2} f^\prime\left(\frac 1{n-1} \right)=:
- 2 \delta  >  F^\prime\left(\frac1{6k}\right) \ge
F^\prime\left( \frac12 \right) =  -\frac{2}{ \pi^2 k^2} f^\prime(0)=-2.
\]
Otherwise, for $1/(6k) > 1/2$, $g^\prime$ is increasing from $-2$ to $-2 \delta^\prime$ on $[0,1/2]$, and $F^\prime(0) =:
- 2 \delta  >
F^\prime\left(1/2 \right) = -2$, where $\delta,\delta^\prime\in(0,1)$.

Therefore, there is only one point $x^\ast\in \left(0,\min\big(\frac12,\frac1{6k}\big)\right)$ such that $S^\prime(x^\ast)=g^\prime(x^\ast)-F^\prime(x^\ast)=0$. This gives the unique solution for \eqref{eq:nonT-27-nonBVP}.

In the case of $k\in\R\backslash\{0\}$, we use $|k|$ instead of $k$ and  mention that
%S Aug.10,          ^^ improved from \Z
\[
g(x)=x - \frac{\sin(2 \pi k x)}{\pi k} = x - \frac{\sin(2 \pi |k| x)}{\pi |k|}.
\]
Hence the proof is finished.
% ^^^^^^^^^^^^^^^^^^^^^^^^^^^^^^^^^^^^
\end{proof}

This enables us to give sharp estimates on the worst case error for the full space $H^1([0,1])$.

\begin{thm}\label{thm5}

In the case of $H^1([0,1])$ with  $k\in\R\backslash\{0\}$,
%S Aug.10,                             ^^ improved from \Z
the radius of information is given by
\[
r(N) = \frac 1{2 \pi |k|}\left( L_0 -  \frac{\sin(2 \pi k L_0)}{\pi k} + L_n -
\frac{\sin(2 \pi k L_n)}{\pi k} + 1 - \frac{1}{ \pi^2 k^2}\sum_{j=1}^{n-1}\frac{\sin^2(\pi k L_j)}{L_j}\right) ^{1/2},
\]
where $L_0=x_1,\ L_j= x_{j+1}-x_j,\ j=1,\ldots,n-1$ and $L_n= 1 - x_n$, with $0\le x_1 < \dots < x_n \le 1$.

Moreover, if $n-1 \ge 2.7 |k|$, then $x_1=1-x_n=x^\ast$ with $x^\ast$ from
Lemma \ref{lem:nonT-27-nonBVP},
and equidistant $ x_j = \frac{j-1}{n-1} \cdot (x_n-x_1) +
%S Aug. 11:                    ^^ typo.
x_1,\ j=2,\ldots,n-1$,
are optimal in the worst case.
\qed
% the worst case error
%\[
% e (n, I_{\rho_k}, H^1_0) = \frac1{2\pi |k|} \left(\dots \dots \right)^{1/2}.
%  \]
\end{thm}

\begin{rem}\label{rem:compare-x-ast}
Although we do not give an explicit formula for the point $x^\ast$ above,
it is easy to obtain the numerical solution for $x^\ast$ when $k$ and $n$ are known.
We want also to ask whether equidistant nodes, $x_j=\frac j{n+1},\ j=1,\dots,n$,
are optimal for some $k$ and $n$. The answer is negative. Firstly, it can only happen
if $n+1> 6 |k|$. We take $t= \frac{|k|}{n+1}\in(0,1/6)$ and find that, from Lemma \ref{lem:nonT-27-nonBVP},
\[
S^\prime\left(\frac 1{n+1}\right)= 2-4\cos(2 \pi t) + \frac 2{\pi^2} \frac{\sin(\pi t)}{t^2}
\left( 2 \pi t \cos(\pi t) - \sin(\pi t)  \right) > S^\prime(x^\ast) = 0, \quad\quad t \in \left(0,\frac16\right).
\]
%S We can get the above inequality via drawing a figure.
This tells us that, $x^\ast < \frac 1 {n+1}$ if $n-1> 2.7 |k|$.
Even we have $x^\ast <\frac 1 {2n}$, since for the midpoint rule, i.e.,
$x_j=\frac {2j-1}{2n},\ j=1,\dots,n$,
\[
S^\prime\left(\frac 1{2n}\right)= 2-4\cos( \pi t) + \frac 2{\pi^2} \frac{\sin(\pi t)}{t^2}
\left( 2 \pi t \cos(\pi t) - \sin(\pi t)  \right)>0 , \quad\quad t = \frac{|k|}{n} \in \left(0,\frac13\right).
\]
%S Again, we can get the above inequality via drawing a figure.
That is, the  endpoints nearby
are much closer to the optimal $x_1$ and $x_n$ than $x_2$ and $x_{n-1}$, respectively, with the
distance $x^\ast < \frac 1 {2n} < \frac 1 {n+1} < \frac 1 {n}
< \frac{1- 2 x^\ast}{n-1} = x_{j+1}- x_j < \frac 1 {n-1},\ j=1,\dots,n-1$.
%S Aug.10,                        added ^^^^^^^^^^^^^^^^^
\qed
%S 22 July: For Table \ref{tab:example2}, I remove it from here
%to Section \ref{Se:osci4} since the spline algorithm in Section \ref{Se:osci4} is included.
% And we refer to the connections via one sentence below.
%   Some numerical results are collected in Section \ref{Se:osci4} to compare the worst case errors
%   for four types of equidistant nodes, see Table \ref{tab:example2}.
\end{rem}

%E I moved the next remark into the next Section and changed it a bit.

\section{Oscillatory integrals: equidistant nodes}\label{Se:osci4}

In this section, we want to discuss the case of equidistant nodes
for the Sobolev space $H^1([0,1])$ of non-periodic functions.
% Given the points $x_0,\ldots,x_n$, the spline algorithm is linear and optimal for our integration problem.
% see \cite[Cor. 5.7.1]{TWW88} and \cite[p. 110]{TW80}.
Throughout this section, we assume that one uses equidistant nodes
\begin{equation}\label{equi-nodes}
x_j = \frac{j}{n},  \qquad j = 0, 1, \dots , n.
\end{equation}
This case was already studied by Boltaev et al. \cite{BHS16b},
%S Aug. 11    add ¡°(for Fourier integrals) with $k\in\Z\backslash\{0\}$¡±?
using the S. L. Sobolev's method.

Then the oscillatory integral $I_{\rho_k}$ of the piecewise linear function $\sigma$
(the spline algorithm) is given by
\begin{equation}\label{spline-alg}
A_{n+1}^k(f) = I_{\rho_k}(\sigma) = \sum_{j=0}^n a_j f(x_j) ,
\end{equation}
where the coefficients $a_j$'s are given as follows.
We skip the proof since the result is known, see \cite[Theorem 8]{BHS16b}.

\begin{prop} \label{thm:spline_quadra}
Let $k\in\Z\backslash\{0\},\;n\in\N$, and $x_j = {j}/{n},\; j = 0, 1, \dots, n$.
Assume that $f:\;[0,1]\rightarrow\C$ is an integrable function
with $f(x_0),f(x_1), \dots, f(x_n)$ given, and $\sigma$ is the
piecewise linear function of $f$ at $n+1$ equidistant nodes $\{x_j\}_{j=0}^n$.
Then $I_{\rho_k}(\sigma)= \sum_{j=0}^n a_j f(x_j)$, where
% the coefficients $\{a_j\}_{j=0}^n$ are as follows.
 \[
\begin{split}
a_0 &= \frac n{4k^2\pi^2} \left(1 - \frac{2\pi i k}n  - e^{-2\pi i k/n}\right),\\
a_j &= \frac n{k^2\pi^2} \sin^2\left(\frac{\pi k}n\right)  e^{-2\pi i kj/n},\;\;\; j=1,\ldots,n-1,\\
a_n &= \frac n{4k^2\pi^2} \left(1 + \frac{2\pi i k}n  - e^{2\pi i k/n}\right),
\end{split}
\]
and $\sum\limits_{j=0}^n a_j =0$.
\qed
\end{prop}

\begin{rem}\label{re:equi-QMC}
We comment on the weights $a_j$ in Proposition~\ref{thm:spline_quadra}. Obviously, for every $j=1,\dots,n-1$, we have
$$
\lim_{n\rightarrow\infty} a_j e^{2\pi ikj/n} \cdot n = 1\;\;\;\text{and}\;\;\;
\lim_{n\rightarrow\infty} a_0\cdot n  = \lim_{n\rightarrow\infty} a_n\cdot n = \frac 12.
$$
Therefore, we conclude that for sufficiently large $n$, the linear algorithm
is almost a QMC (quasi Monte Carlo) algorithm with equidistant nodes, which is used in \cite{NUW15}.
\qed
\end{rem}

Clearly,  from Theorem~\ref{thm3}, the algorithm $A^{k}_{n+1}$ with equidistant nodes is optimal
for the space $H^1_0$ in the worst case if $n\ge 2.7 |k|$. Here,
$n$ stands for the number of the intervals.
Boundary values are fixed
for $f\in H^1_0([0,1])$, i.e., $f(0)=f(1)=0$.
% and $f(0)=f(1)$ for  $f\in \wt H^1([0,1])$.
% Usually in this paper, $A_{n+1}$ uses $n+1$ function values.
% We prefer to use the meaning that this number $n$ is the number of function values.

Furthermore, we have the following assertion for the space $H^1$,
in which the point (i) is already proved in \cite[Theorem 9]{BHS16b}.

\begin{thm} \label{thm6}
Consider the integration problem $I_{\rho_k}$ defined
for functions from the space $H^1([0,1])$.
%S Aug. 11             ^^^^^^ instead of spaces
Suppose $k\in\Z$ and $k\neq 0$.
\qquad
\begin{enumerate}[(i)]
\item The worst case error of $A^{k}_{n+1}$, $n\in\N$, is
$$
e  (A^{k}_{n+1}, I_{\rho_k}, H^1)=
\frac1{2\pi |k|} \left( 1 - \frac{n^2}{k^2\pi^2} \sin^2\left(\frac{ k \pi}{n} \right) \right)^{1/2}.
$$
\item
For $n\in\N$, we have
$$
e  (A^{k}_{n+1}, I_{\rho_k}, H^1) <  \, e (0, I_{\rho_k}, H^1_0) = \frac1{2\pi |k|},\;\;\text{if}\;\;\; k \neq 0\;\text{mod}\; n.
$$
\item
For fixed $n\in\N$, we have
\[
\lim\limits_{|k|\rightarrow\infty} e  (A^{k}_{n+1}, I_{\rho_k}, H^1)\,\cdot\,  |k| =  \frac{1}{2\pi} .
%S Aug.9 ^^^^^^
\]
\item
For any $k\in\Z\backslash\{0\}, n\in\N$, we have
\[
e  (A^{k}_{n+1}, I_{\rho_k}, H^1) \le \frac 1{2\sqrt 3}\, \frac 1{n}.
\]
\item
For fixed $k\in\Z\backslash\{0\}$, we have the sharp constant of asymptotic equivalence $\frac 1{2\sqrt 3}$, i.e.,
\[
\lim\limits_{n\rightarrow\infty}e  (A^{k}_{n+1}, I_{\rho_k}, H^1) \,\cdot\, n = \frac 1{2\sqrt 3}.
\]
% \item
% Let $\a\in(0,1)$. Then for $n> [\a/(1-\a)]\,|k|$ we have
% $$
% \tilde e (A^{k}_{n}) \le \frac 1{2\a\sqrt 3}\,
% \,\frac1{n+|k|}.
% $$
\qed
\end{enumerate}
\end{thm}

\begin{proof}
The point (i) follows from Theorem~\ref{thm3} directly since $N(f)=0$ tells us
that $f(0)=f(1)=0$ and $f\in H_0^1$. Then points (ii) and (iii) follow clearly.
We use Taylor's expansion of the cosine function at zero. For any $k\in\Z\backslash\{0\}, n\in\N$,
\[
\sin^2\left(\frac{k\pi}{n}\right) = \frac {1 -\cos\left(\frac{2k\pi}{n}\right)}2
=\frac{k^2\pi^2}{n^2} -\frac12 R_3\left(\frac{2k\pi}{n}\right)
%,\;\; \text{for some} \;\; \theta=\theta\left(\frac{2k\pi}{n+1}\right) \in (0,1)
.
\]
Here, the third Lagrange's remainder term satisfies, for some $\theta=\theta\left(\frac{2k\pi}{n}\right) \in (0,1)$,
\[
\abs{R_3 \left(\frac{2k\pi}{n} \right)} = \abs{\cos^{(4)} \left(\theta \cdot \frac{2k\pi}{n} \right)}
\cdot \frac{({2k\pi})^4}{4!\cdot n^4} \le \frac 23 \left(\frac{k\pi}{n} \right)^4.
\]
This implies that, for any $k\in\Z\backslash\{0\}, n\in\N$,
\[
0 <
1 - \frac{n^2}{k^2 \pi^2} \sin^2\left(\frac{ k \pi}{n} \right) =
 \frac{n^2}{2 k^2 \pi^2} \cdot \abs{R_3\left(\frac{2k \pi}{n}\right)}  \le \frac 13 \left(\frac{k \pi}{n} \right)^2.
\]
Hence,  for any $k\in\Z\backslash\{0\}, n\in\N$,
\[
e  (A^{k}_{n+1}, I_{\rho_k}, H^1) \le \frac 1{2\sqrt 3}\, \frac 1{n}.
\]
This proves (iv).

Moreover, if $k$ is fixed and nonzero, we have that for any $\theta\in(0,1)$,
\[
\lim_{n\rightarrow\infty}\cos^{(4)}\left(\theta \cdot \frac{2k\pi}{n}\right)=1.
\]
 This leads to
\[
\lim\limits_{n\rightarrow\infty} e (A^{k}_{n+1}, I_{\rho_k}, H^1) \,\cdot\, n = \frac 1{2\sqrt 3},
\]
as claimed in (v).
\end{proof}

%S 22 July: I give the detailed comments below
% for explaining the latter statement in Remark \ref{re:equi-dist2}.

We comment on Theorems~\ref{thm3} and~\ref{thm6}.
Theorem~\ref{thm6} deals with $k\in\Z\backslash\{0\}$ and equidistant nodes, while Theorem~\ref{thm3} works even for $k\in\R\backslash\{0\}$.
However, Theorem~\ref{thm3} studies only the space $H^1_0$ instead of $H^1$.

For $k\in\R\backslash\{0\}$, the same statements, as in Theorem~\ref{thm6},
hold true for the space $H^1_0$, since the spline algorithm is optimal.
Due to the zero boundary values,
the number of information is $n-1$ for  $H^1_0$, instead of $n+1$. This is indeed a special case of Theorem~\ref{thm3}.

Moreover, thanks to the equidistant nodes including endpoints, the formula in point (i) of Theorem~\ref{thm6} remains valid
for $k\in\R\backslash\{0\}$ (and $H^1$), as well as points (iii)-(v).
% For the space $H^1$ and the optimal algorithm (using splines) $A^{k}_{n+1}$,
% we consider the worst case error by $r(N)$.
%  The information nodes
%  in Theorem~\ref{thm6} include the endpoints due to the equidistance way.
In the computation of $r(N, H^1)$, we usually work with
\[
N(f)=\left(f(0),f\left(\frac 1n \right),\dots, f\left(\frac{n-1}n\right), f(1) \right)=0 \quad  \text{for} \quad f \in H^1([0,1]).
\]
This is equivalent to the computation of $r(N_1, H^1_0)$ in Theorem~\ref{thm3} with
\[
N_1(f)=\left(f\left(\frac 1n \right), \dots, f\left(\frac{n-1}n \right) \right)=0 \quad  \text{for} \quad f \in H^1_0([0,1]).
\]
%which is familiar to us in Theorem~\ref{thm3}.
That is shortly, for $k\in\R\backslash\{0\}$,
\[
\begin{split}
e  (A^{k}_{n+1}, I_{\rho_k}, H^1)&= r(N, H^1) =
\sup_{f\in H^1:\, \Vert f \Vert \le 1 \atop \ N(f)=0}  | I_\rho (f) |=
\sup_{f\in H^1_0:\, \Vert f \Vert \le 1 \atop \ N_1(f)=0}  \abs{ I_\rho (f) }\\
&=
r(N_1, H^1_0) =
\frac1{2\pi |k|} \left( 1 - \frac{n^2}{k^2\pi^2} \sin^2\left(\frac{ k \pi}{n} \right) \right)^{1/2}.
\end{split}
\]

%E I like Theorem 17, also the asymptotic results (iii) and (v).
%E We could mention these results even in the intro, where we give a short
%E summary of all results.
%E In Theorem 17, we use the "optimal algorithm based on equidistant nodes".
%E We already know that the statement (v) is also true if we use optimal
%E nodes instead of equidistant nodes.
%E I guess it is also clear and easy to prove that (iii) is true if we replace equidistant
%E nodes by optimal nodes. We should state this! What do you think?
%S That is true, surely.

\begin{rem}
It is easy to prove that these asymptotic statements (iii) and (v)
also hold for optimal nodes, i.e., for the numbers
$e(n, I_{\rho_k}, H^1)$ with $k\in\R\backslash\{0\}$.
%E  Do you agree?  Should we say more?
%S Aug.11 Sure. I agree and want to say more here.
% (iii) for arbitrary nodes and the space H^1_0??? The same constant 1/(2\pi), see Cor. 10 (ii).
%  maybe some reader has misunderstandings for the notation $e(n, I_{\rho_k}, H^1)$?
% Sometimes it focuses on optimal nodes and algorithms, both.
% Sometimes, we use it only for fixed equidistant nodes ( but not optimal nodes ).
More precisely,  for fixed $n$ and $k\rightarrow \infty$, one can take $L_0=L_n=0$ in Theorem~\ref{thm5} to get the asymptotic property of
$e(n, I_{\rho_k}, H^1)$.
For fixed $k\in\R\backslash\{0\}$ and $n\rightarrow \infty$, Theorem~\ref{thm5} gives by Taylor's expansions the same asymptotic constant for
$e(n, I_{\rho_k}, H^1)$ since $x^\ast < \frac 1 {2n}$ and $\frac 1 {n}< \frac{1- 2 x^\ast}{n-1} < \frac 1 {n-1}$.
Finally, together with Corollary~\ref{cor: asym-2sq3}, we find out the same asymptotic
constants, $1/(2 \pi)$ and $1/(2\sqrt 3)$, for both the spaces $H_0^1$ and $H^1$.
%E   (August)                              ^^^^^^^^
%S Aug. 31 Agreed.
\end{rem}

\noindent
\subsection*{Acknowledgement}
This
work was started while S. Zhang was visiting Theoretical Numerical Analysis Group
at Friedrich-Schiller-Universit\"at Jena. He is extremely grateful for their hospitality.


\begin{thebibliography}{99}

\setlength{\parsep }{-0.5ex}
\setlength{\itemsep}{-0.5ex}

\frenchspacing

\newcommand\BAMS{\emph{Bull. Amer. Math. Soc.\ }}
\newcommand\BIT{\emph{BIT\ }}
\newcommand\Com{\emph{Computing\ }}
\newcommand\CA{\emph{Constr. Approx.\ }}
\newcommand\FCM{\emph{Found. Comput. Math.\ }}
\newcommand\JAT{\emph{J. Approx. Th.\ }}
\newcommand\JC{\emph{J. Complexity\ }}
\newcommand\JMA{\emph{SIAM J. Math. Anal.\ }}
\newcommand\JMAA{\emph{J. Math. Anal. Appl.\ }}
\newcommand\JMM{\emph{J. Math. Mech.\ }}
\newcommand\MC{\emph{Math. Comp.\ }}
\newcommand\NM{\emph{Numer. Math.\ }}
\newcommand\RMJ{\emph{Rocky Mt. J. Math.\ }}
\newcommand\SJNA{\emph{SIAM J. Numer. Anal.\ }}
\newcommand\SR{\emph{SIAM Rev.\ }}
\newcommand\TAMS{\emph{Trans. Amer. Math. Soc.\ }}
\newcommand\TOMS{\emph{ACM Trans. Math. Software\ }}
\newcommand\USSR{\emph{USSR Comput. Maths. Math. Phys.\ }}

\frenchspacing

\addcontentsline{toc}{chapter}{Bibliography}

%  \bibitem{A92}
%  M. Atteia,
%  \emph{Hilbertian Kernels and Spline Functions},
%  North-Holland, Amsterdam, 1992.

% \bibitem{B59}
% N. S. Bakhvalov,
% On approximate computation of integrals,
% \emph{Vestnik MGU, Ser. Math. Mech. Astron. Phys. Chem},
% \textbf{4}, 3--18, 1959, in Russian.

%  \bibitem{B71}
%  N. S. Bakhvalov,
%  On the optimality of linear methods for operator approximation
%  in convex classes of functions,
%  \emph{USSR Comp. Math. Math. Phys.}
%  {\bf 11}, 244--249, 1971.

%  \bibitem{BT04}
%  A. Berlinet and C. Thomas-Agnan,
%  \emph{Reproducing Kernel Hilbert Spaces in Probability
%  and Statistics}, Kluwer, Boston, 2004.

%E next deleted
%  \bibitem{BHS16a}
%  N. D. Boltaev, A. R. Hayotov and Kh. M. Shadimetov,
%  Construction of optimal quadrature formula for numerical
%  calculation of Fourier coefficients in Sobolev space
%  $L^{(1)}_2$,
%  \emph{American Journal of Numerical Analysis}  {\bf 4}, 1--7, 2016.

\bibitem{BHS16b}
N. D. Boltaev, A. R. Hayotov and Kh. M. Shadimetov,
Construction of optimal quadrature formula for Fourier coefficients in Sobolev space
$L^{(m)}_2(0,1)$,
\emph{Numerical Algorithms}, in press, DOI: 10.1007/s11075-016-0150-7, 2016.

%  \bibitem{Ca99}
%  G. Cain, \emph{Complex Analysis}, textbook,
%  \newline
%  http://people.math.gatech.edu/$\sim$cain/winter99/complex.html, 1999.

%  \bibitem{DGK12}
%  V. Dom\'\i nguez, I. G. Graham and  T. Kim,
%  Filon-Clenshaw-Curtis rules for highly-oscillatory integrals
%  with algebraic singularities and stationary points,
%  arXiv: 1207.2283v1. July 2012.

%  \bibitem{CW79}
%  P. Craven and G. Wahba,
%  Smoothing noisy data with spline functions,
%  \NM {\bf 31}, 377--403, 1979.

%  \bibitem{DGS11}
%  V. Dom\'\i nguez, I. G.  Graham and V. P.  Smyshlyaev,
%  Stability and error estimates for Filon-Clenshaw-Curtis rules for highly
%  oscillatory integrals,
%  \emph{IMA Journal of Numerical Analysis} {\bf 31}, 1253--1280, 2011.

%  \bibitem{HNUW12}
%  A. Hinrichs, E. Novak, M. Ullrich, H. Wo\'zniakowski,
%  The Curse of Dimensionality for Numerical Integration
%  of Smooth Functions,
%  \emph{submitted}, 2012.

%  \bibitem{HNW11}
%  A. Hinrichs, E. Novak, H. Wo\'zniakowski,
%  The curse of dimensionality for the class of monotone
%  functions and for the class of convex functions,
%  \emph{J. Approx. Th.} \textbf{163}, 955--965, 2011.

\bibitem{BP11}
H. Brass and K. Petras, \emph{Quadrature Theory: The Theory of Numerical Integration
on a Compact Interval}, AMS Mathematical Surveys and Monographs,  Vol. {\bf 178}, 363 pp, Rhode Island, 2011.

\bibitem{Cia02}
P. G. Ciarlet,
\emph{The Finite Element Method for Elliptic Problems}, Classics in applied
mathematics  {\bf 40}, Society for Industrial and Applied Mathematics, Philadelphia, 2002.

%E  This is a nice survey, I guess.
%S Aug. 31 Agreed.
\bibitem{HO09}
D. Huybrechs and S. Olver,
Highly oscillatory quadrature,
Chapter 2 in: \emph{Highly Oscillatory Problems},
London Math. Soc. Lecture Note Ser. {\bf 366}, Cambridge, pp. 25--50, 2009.

%  \bibitem{IN05}
%  A. Iserles and S. N\o rsett,
%  Efficient quadrature of highly oscillatory integrals
%  using derivatives,
%  \emph{Proc. R. Soc. Lond. Ser. A Math. Phys. Eng. Sci.}
%  {\bf 461}, 1383--1399, 2005.

%  \bibitem{KYWNT06}
%  Y. Kametaka, H. Yamagishi, K. Watanabe, A. Nagai and K. Takemura,
%  Riemann zeta function, Bernoulli polynomials and the best constant
%  of Sobolev inequality,
%  \emph{Scientiae Mathematicae Japonicae} Online {\bf e-2007}, 63--89, 2006.

\bibitem{Lan67}
 H. J. Landau, Necessary density conditions for sampling and interpolation of certain entire functions,
 \emph{Acta Math.}\ {\bf 117}(1): 37--52,
% doi:10.1007/BF02395039,
 1967.

%  \bibitem{La06}
%  J. Lawson, \emph{Course Syllabi of Differential Geometry} (Chapter 4), 2006.
%  \newline  (https://www.math.lsu.edu/$\sim$lawson/Chapter4.pdf)

%  \bibitem{Me10}
%  J. M. Melenk,
%  On the convergence of Filon quadrature,
%  \emph{J. Comput. Applied Math.} {\bf 234}, 1692--1701, 2010.

\bibitem{LM54}
P. D. Lax and A. N. Milgram, Parabolic equations, \emph{Ann. of Math.} {\bf 33},
167--190, 1954.

\bibitem{ME09}
M. Mishali and Y. C. Eldar, Blind multiband signal reconstruction: compressed sensing
for analog signals, \emph{IEEE Trans. Signal Processing}\ {\bf 57}(3), CiteSeerX: 10.1.1.154.4255, 2009.

%  \bibitem{No88}
%  E. Novak,
%  \emph{Deterministic and Stochastic Error Bounds in
%  Numerical Analysis},
%  LNiM {\bf 1349}, Springer-Verlag, Berlin, 1988.

%  \bibitem{NSW04}
%  E. Novak, I. Sloan and H. Wo\'zniakowski,
%  Tractability of Approximation for Weighted Korobov Spaces
%  on Classical and Quantum Computers,
%  \emph{Found. Comput. Math.} {\bf 4}, 121--156, 2004.

\bibitem{NUW15}
E. Novak, M. Ullrich and H. Wo\'zniakowski,
Complexity of oscillatory integration for univariate Sobolev spaces,
\emph{J. Complexity} {\bf 31}, 15--41, 2015.

\bibitem{NUWZ15}
E. Novak, M. Ullrich, H. Wo\'zniakowski and S. Zhang,
Complexity of oscillatory integrals on the real line,
submitted, arXiv: 1511.05414 [math. NA],  2015.

%  \bibitem{NW01}
%  E. Novak and H. Wo\'zniakowski,
%  When are integration and discrepancy tractable?
%  In: \emph{Foundations of Computational Mathematics}. R. A. DeVore, A. Iserles, E. S${\rm \ddot{u}}$li (eds),
%  Cambridge University Press, 211-266, 2001.

% \bibitem{NW08}
% E. Novak and H. Wo\'zniakowski,
% \emph{Tractability of Multivariate Problems},
% Volume I: Linear Information,
% European Math. Soc. Publ. House, Z\"urich,
% 2008.

%  \bibitem{NW09}
%  E. Novak and H. Wo\'zniakowski,
%  Approximation of infinitely differentiable multivariate functions
%  is intractable,
%  \JC  \textbf{25}, 398--404, 2009.

%  \bibitem{NW10}
%  E. Novak and H. Wo\'zniakowski,
%  \emph{Tractability of Multivariate Problems},
%  Volume II: Standard Information for Functionals,
%  European Math. Soc. Publ. House, Z\"urich,  2010.

%  \bibitem{Ol08}
%  S. Olver,
%  \emph{Numerical Approximation of Highly Oscillatory Integrals},
%  dissertation, University of Cambridge, 2008.

%  \bibitem{Rv90}
%  V. A. Rvachev,
%  Compactly supported solutions of functional-differential
%  equations and their applications,
%  \emph{Russian Math. Surveys} {\bf 45}, 87-120, 1990.

%  \bibitem{Sa77}
%  G. Sansone,
%  \emph{Orthogonal Functions}, 2nd ed. John Wiley and Sons Inc, New York,
%  1977.

%  \bibitem{SW02}
%  I. H. Sloan and H. Wo\'zniakowski,
%  Tractability of integration in non-periodic and periodic
%  weighted tensor product Hilbert spaces, J. Complexity, 18, 479--499, 2002.

%  \bibitem{Su79}
%  A. G. Sukharev,
%  Optimal numerical integration formulas
%  for some classes of functions of several variables,
%  \emph{Soviet Math. Dokl.}  \textbf{20}, 472--475, 1979.

%  \bibitem{S65}
%  S. A. Smolyak,
%  \emph{On optimal restoration of functions and functionals of them},
%  (in Russia), Candidate Dissertation, Moscow State University, 1965.

%  \bibitem{SW71}
%  E. M. Stein and G. Weiss,
%  \emph{Introduction to {F}ourier analysis on {E}uclidean spaces},
%  Princeton Mathematical Series, No. 32,
%  Princeton University Press, Princeton, N.J., 1971.

% \bibitem{Ta96}
% C. Thomas-Agnan, Computing a family of reproducing kernels for statistical applications,
% \emph{Numerical Algorithms}  {\bf 13},  21--32, 1996.

\bibitem{TWW88}
J. F. Traub, G. W. Wasilkowski and H. Wo\'zniakowski,
\emph{Information-Based Complexity},
Academic Press, 1988.

\bibitem{TW80}
J. F. Traub  and H. Wo\'zniakowski,
\emph{A General Theory of Optimal Algorithms},
Academic Press, 1980.

%  \bibitem{MU14}
%  M. Ullrich,
%  On ``Upper error bounds for quadrature formulas on function classes'' by K. K. Frolov,
%  \emph{ArXiv e-prints} 1404.5457, 2014.

% \bibitem{W90}
% G. Wahba,
% \emph{Spline models for observational data},
% CBMS-NSF Regional Conference Series in Applied Mathematics, 59,
% Society for Industrial and Applied Mathematics (SIAM),
% Philadelphia, PA, 1990.

\bibitem{Zen77}
A. A. ${\rm\check{Z}}$ensykbaev,	
Best quadrature formula for some classes of periodic differentiable functions,
\emph{Izv. Akad. Nauk SSSR Ser. Mat.}\	{\bf 41}(5), 1110--1124, 1977 (in Russian);
English transl.,
Math. USSR Izv.\ {\bf 41}(5), 1055--1071, 1977.

\end{thebibliography}
\end{document}